\documentclass{amsart}
\usepackage[active]{srcltx}
\usepackage{setspace,graphicx,graphics,amsmath,chngpage,amsthm}
\usepackage{amsmath, array,amsfonts,amssymb,newlfont,hhline}

\newcommand{\D}{\mathcal{D}}

\newcommand{\Real}{\mathbb R}
\newcommand{\abs}[1]{\left\vert#1\right\vert}

\newtheorem{lem}{Lemma}
\newtheorem{thm}{Theorem}

\newtheorem{rem}{Remark}

\numberwithin{thm}{section}
\numberwithin{lem}{section}
\numberwithin{coll}{section}
\numberwithin{rem}{section}
\numberwithin{exm}{section}
\numberwithin{prop}{section}
\numberwithin{equation}{section}

\numberwithin{equation}{section}

\begin{document}

\centerline {\textsc {\large On the convergence of entropy for $K^{th}$ extreme}}

\vspace{0.5in}

\begin{center}
   Ali Saeb \footnote{Corresponding author: ali.saeb@gmail.com}\\
	Department of Economic Sciences,\\
  Indian Institute of Science Education and Research,
   Bhopal 462 066, India\\
   \end{center}

\vspace{1in}

\noindent {\bf Abstract:} Let recall that the term "$k^{th}$ extreme" was introduced in a limiting sense. That is, if $X_{r:n}$ denote the $r^{th}$ order statistic then for fix $k,$ as $n\to\infty$, $X_{n-k+1:n}$ is called the $k^{th}$ extremes or $k^{th}$ largest order statistics. 
In this paper, we study entropy limit theorems for $k^{th}$ largest order statistics under linear normalization. We show the necessary and sufficient conditions which convergence entropy of $k^{th}$ extreme holds.
\vspace{0.5in}

\vspace{0.2in} \noindent {\bf Keywords:} Entropy convergence, Max domains of attraction, $K^{th}$ Extreme.

	\vspace{0.5in}
\vspace{0.2in} \noindent {\bf MSC 2010 classification:} 60F10
\newpage
\section{Introduction}
Shannon (1948) defines entropy of an absolutely continuous random variable (rv) $X$ with distribution function (df) $F$ and probability density function (pdf) $f$ as
\[
h(f) = E(- \log f(X)) = -\int_{\Real}
f(x) \log f(x)dx
\]
with the convention that the integral is over all real values for which the density is positive. 
The definitions of entropy and some of their consequences are studied by Lazo and Rathie (1978), Ebrahimi et al. (1999), and Ravi and Saeb (2014b).

The limit laws of linearly normalized partial maxima $M_n=X_1\vee\ldots\vee X_n$ of iid rvs $X_1,X_2,\ldots,$ with common distribution function $F,$ namely,
\begin{equation}\label{Introduction_e1}
	\lim_{n\to\infty}P(M_n\leq a_nx+b_n)=\lim_{n\to\infty}F^n(a_nx+b_n)=G(x),\;\;x\in \mathcal{C}(G),
\end{equation}
where, $a_n>0,$ $b_n\in\Real,$ are norming constants, $G$ is a non-degenerate df, $\mathcal{C}(G)$ is the set of all continuity points of $G,$ are called max stable laws. If, for some $G,$ a df $F$ satisfies (\ref{Introduction_e1}) for some norming constants $a_n>0,$ $b_n\in\Real,$ then we say that $F$ belongs to the max domain of attraction of $G$ under linear normalization and denote it by $F\in \mathcal{D}(G).$ Limit dfs $G$ satisfying (\ref{Introduction_e1}) are the well known extreme value types of distributions, or max stable laws, namely,
	\begin{center}
	\begin{tabular}{c l}
		the Fr\'{e}chet law: & $\Phi_\alpha(x)=\left\lbrace	
							\begin{array}{l l}
							 0, &\;\;\; x< 0, \\
							 \exp(-x^{-\alpha}), &\;\;\; x\geq 0;\\
							 \end{array}
							 \right.$ \\
\\							
		the Weibull law: & $\Psi_\alpha(x)=\left\lbrace
						\begin{array}{l l} 0, & x<0, \\
						\exp(- \mid x \mid^{\alpha}), & 0\leq x;
						\end{array}\right.$ \\
\\
		and the Gumbel law: & $\Lambda(x)=\exp(-\exp(-x));\;\;\;\;\; x\in\Real;$\\
	\end{tabular}
\end{center}
$\alpha>0$ being a parameter, with respective densities
$\phi_\alpha(x)= 0,$ if $x< 0,$ and $=\alpha x^{-\alpha-1}e^{-x^{-\alpha}}$ if $x > 0$ is the Fr\'echet density; $\psi_\alpha(x)= 0,$ if $x>0,$ and $=\alpha (-x)^{\alpha-1}e^{-(-x)^{\alpha}},$ if $x<0$ is the Weibull density and  
 $\lambda(x)=e^{-x}e^{-e^{-x}},\;\; x\in\Real$ is the Gumbel density.
Note that (\ref{Introduction_e1}) is equivalent to
\begin{eqnarray}\label{Introduction_e1.2}
\lim_{n\to\infty}n(1-F(a_nx+b_n))=-\log G(x), \; x \in \{y: G(y) > 0\}.
\end{eqnarray}
Let $X_{r:n}$ denote the $r^{th}$ order statistic then for fix $k,$ as $n\to\infty$, $X_{n-k+1:n}$ is called the $k^{th}$  extremes. The distribution functions of $X_{n-k+1:n}$ take very simple forms in the case of iid rvs. It is also well known that if the df $F$ satisfies (\ref{Introduction_e1}) for some $G,$ then the $k^{th}$ largest of $\{X_1,\ldots,X_n\}$ 
converges with the same normalization to a nondegenerate df. More precisely, if (\ref{Introduction_e1}) holds for some norming constants $a_n$, $b_n$ and some dfs $F$ and $G$ then
\begin{eqnarray}\label{Introduction_e1.3}
\lim_{n\to\infty} P\left( X_{n-k+1:n} \leq a_n x + b_n \right) = G(x) \sum_{i=0}^{k-1} \frac{(- \log G(x))^i}{i!}, \;
\end{eqnarray}
for  $x \in \{y: G(y) > 0\},$ (see, Galambos 1987). Suppose that (\ref{Introduction_e1}) holds for some df $F$ and some max stable law $G.$ Let $f$ and $g$ respectively denote the pdfs of $F$ and $G.$ Let $g_n(x)=na_nf(a_nx+b_n)F^{n-1}(a_nx+b_n)$ is pdfs of the rv $\left( \dfrac{M_n-b_n}{a_n} \right),$ the pdf of $X_{n-k+1:n}$ is given by
\begin{eqnarray}\label{Introduction_e2}
g_n^{(k)}(x)&=& \frac{n!}{(k-1)!(n-k)!} a_n f(a_nx+b_n)F^{n-k}(a_nx+b_n)(\bar{F}^{k-1}(a_nx+b_n))\nonumber\\
&=& \frac{F^{n-k}(a_nx+b_n)} {n^{k-1}B(n,k)}a_nf(a_nx+b_n)(n\bar{F}(a_nx+b_n))^{k-1},\;\;x\in\Real,\;\;n\geq 1.
\end{eqnarray}
where, $\bar{F}(a_nx+b_n)=1-F(a_nx+b_n)$ and $B(n,k)=Beta(n-k+1,k)=\frac{(k-1)!(n-k)!}{n!}.$ The limit law follows one of the laws of the Fr\'echet, the Weibull, and the Gumbel studied by Hall (1978).

 The entropy convergence of domain attraction is an interesting subject for study in between application investigations. Linnik (1959) and Shimizu (1975) suggested applying the Shannon entropy convergence for investigating the central limit theorem. Brown (1982), Barron (1986), and Takano (1987) were the first to demonstrate a central limit theorem with convergence in the Shannon entropy sense. Cardone et. al. (2022) study the entropy of central limit theorem for order statistics. 
With similar talks in the stable laws and central limit theorems the subject of max stable laws and max domain attraction is established.  Ravi and Saeb (2012) derive entropies of $\ell$-max and p-max families and also of associated
generalized Pareto, generalized log-Pareto and related distributions.  Ravi and Saeb (2014a) study the necessary condition for the Shannon
entropy convergence of the $k^{th}$ extremes. They show that if $f (x) > 0$ is a nonincreasing pdf for $x$ close to $r(F),$ then the
convergence of entropy of $k^{th}$ max rv exists. Saeb(2023b) find the sufficient condition entropy convergence for max stable laws and max domain attraction. He present that if 
$E(-\log f(X))$ is exist then the entropy convergence holds. The properties of the R\'enyi entropy of max domain attraction are studied by Saeb (2018, 2023a).

In this article, we first derive the entropy of $k^{th}$ extreme for the max stable laws. Then we prove the sufficient condition for entropy convergence holds in this case. We should note that, our results satisfy for $k\geq 2.$ In the next section we give our some preliminary and main results, followed by a section on Proofs. 
Next we give an appendix containing the results used in this article.
 
We shall denote the left extremity of df $F$ by $l(F) = \inf\{x: F(x) > 0\} \geq - \infty $ and the right extremity of $F$ by $r(F) = \sup \{x: F(x) < 1\} \leq \infty.$ $\frac{\partial}{\partial x}f(x)$ is partial derivative of $f$ with respect to variable of $x.$ We will refer to Shannon's entropy as entropy in this article.

\section{Main Results}

From definition of entropy and  (\ref{Introduction_e2})  we write
\begin{eqnarray}
	h(g_n^{(k)})&=& -\int_A g_n^{(k)}(x)\log (g_n^{(k)}(x))\;dx, \;\;\mbox{where}\;\;A=\{x\in\Real:g_n^{(k)}(x)>0\}\nonumber\\
 & = &  -\int_A  g_n^{(k)}(x) \left[\log (F^{n-k}(a_nx+b_n)\bar{F}^{k-1}(a_nx+b_n))-\log( B(n,k))\right] dx \nonumber\\
 && - \int_A g_n^{(k)}(x) \log (a_nf(a_nx+b_n))dx\nonumber\\
 & = & - (I_1(n) + I_2(n)), \; \mbox{say. }\label{e0}
 \end{eqnarray}
We then have
\begin{eqnarray}
\lim_{n \rightarrow \infty} h(g_n^{(k)}) = \log \Gamma(k)+k+(k-1)\left(\gamma-\sum_{i=1}^{k-1}\frac{1}{i}\right) - \lim_{n\rightarrow \infty} I_2(n) \;  \label{1.5}
\end{eqnarray}
for $k\geq 2$, in view of the following lemma.

\begin{lem}\label{Lemma1}
\begin{eqnarray}
\lim_{n\to\infty}I_1(n) =-\log \Gamma(k)-k-(k-1)\left(\gamma-\sum_{i=1}^{k-1}\frac{1}{i}\right).\nonumber
\end{eqnarray}
where, $\gamma$ is the Euler's constant.
\end{lem}

Now we need the following lemmata to prove the result on entropy convergence for $k^{th}$ extremes, the first of which gives local uniform convergence for the pdf of $k^{th}$ extreme. 
\begin{lem}\label{Lemma.5} 
Suppose $F$ be absolutely continuous with nonincreasing pdf $f$. If $F\in\D(G)$ for some nondegenerate df $G,$ with norming constants $a_n$ and $b_n$ so that (\ref{Introduction_e1}) holds then the pdf $g_n^{(k)}(x)$ in (\ref{Introduction_e2}) locally uniformly converges to the pdf of $g^{(k)}(x)=g(x)\dfrac{(-\log G(x))^{k-1}}{(k-1)!}.$
\end{lem}
\begin{rem} The density function of $g^{(k)}$ in Lemma \ref{Lemma.5} when $G$ is
\begin{eqnarray}
\text{Fr\'echet: }\phi_\alpha^{(k)}(x)&=&
\begin{cases}
    0 & \text{ if } x\leq 0,\\
    \dfrac{\alpha}{(k-1)!} x^{-\alpha k-1}e^{-x^{-\alpha}}  & \text{ if } x > 0;
\end{cases} \nonumber \\  
\text{ Weibull: }\psi_\alpha^{(k)}(x)&=&
\begin{cases}
\dfrac{\alpha}{(k-1)!} (-x)^{\alpha k-1}e^{-(-x)^{\alpha}}, & \text{ if } x<0,\\
   0  & \text{ if }x \geq 0;
\end{cases} \nonumber \\  
\text{Gumbel: }\lambda^{(k)}(x)&=&\dfrac{1}{(k-1)!} e^{-kx}e^{-e^{-x}},\;\; x\in\Real.\nonumber   
\end{eqnarray}
 \end{rem}
\begin{lem}\label{TPhi_orderEnt}
The entropy of $g^{(k)}$ in (\ref{Introduction_e1.3}) when $G$ and $k\geq 2$ is
\begin{itemize}\label{ent_order}
\item[(i)] Fr\'{e}chet law is $h(\phi_{\alpha}^{(k)})=-\log\dfrac{\alpha}{(k-1)!}-\dfrac{\alpha k+1}{\alpha}\left(-\gamma+\sum_{i=1}^{k-1}\dfrac{1}{i}\right)+k;$
\item[(ii)] Weibull law is
		$h(\psi_\alpha^{(k)})=-\log\dfrac{\alpha}{(k-1)!}-\dfrac{\alpha k-1}{\alpha}\left(-\gamma+\sum_{i=1}^{k-1}\dfrac{1}{i}\right)+k;$
\item[(iii)] Gumbel law is
		$h(\lambda^{(k)})=\log(k-1)!-k\left(-\gamma+\sum_{i=1}^{k-1}\dfrac{1}{i}\right)+k.$
\end{itemize}
	\end{lem}
\begin{thm} \label{TPhi_order}
Let $F\in\D(G)$ for some nondegenerate df $G,$ with norming constants $a_n$ and $b_n$ so that (\ref{Introduction_e1.3}) 
holds. The df $F$ be absolutely continuous with nonincreasing pdf $f$ which is eventually positive, that is, $\;f(x) > 0\;$ for $x$ close to $r(F).$ If $f$ is nonincreasing and enropy of $f$ exists for some $k\geq 2$ and 
\begin{enumerate}
\item[(i)]	$G=\Phi_\alpha$ for some $\;\alpha > 0\;$ then $\;\lim_{n \rightarrow \infty} h(g_n^{(k)}) = h(\phi_\alpha^{(k)});$
\item[(ii)]	$G=\Psi_\alpha$ for some $\;\alpha > 0\;$  then $\;\lim_{n \rightarrow \infty} h(g_n^{(k)}) = h(\psi_\alpha^{(k)});$
\item[(iii)] $G=\Lambda,$ then $\;\lim_{n \rightarrow \infty} h(g_n^{(k)}) = h(\lambda^{(k)}).$
\end{enumerate}
\end{thm}
\section{Proofs}
\begin{proof}[\textbf{Proof of Lemma \ref{Lemma1}}] Making the change of variable $F(a_n x + b_n) = t,\;$ we get $\;a_n f(a_n x + b_n) dx = dt\;$ and from (\ref{e0})
\begin{eqnarray}
I_1(n)  &=& \log(n^{k-1})-\log(n^{k-1} B(n,k))\nonumber\\
&&+\frac{1}{B(n,k)}\int_{0}^{1}  t^{n-k}(1-t)^{k-1} \left[\log t^{n-k}+\log (1-t)^{k-1})\right] \,dt\nonumber\\
&=& \log\left(\dfrac{n!}{(k-1)!(n-k)!n^{k-1}}\right)+(k-1)\log(n) \nonumber\\
&&+\frac{1}{B(n,k)}\left[\int_{0}^{1}  \frac{\partial t^{n-k}}{\partial(n-k)}(n-k)(1-t)^{k-1}\,dt
+\int_{0}^{1}  \frac{\partial (1-t)^{k-1}}{\partial(k-1)}(k-1)t^{n-k}\,dt\right]\nonumber\\
&=& \log\left(\dfrac{(n-k+1)\ldots n}{(k-1)!n^{k-1}}\right)+\frac{1}{B(n,k)}\Big[ (n-k)\frac{\partial}{\partial(n-k)}\left(\int_{0}^{1}t^{n-k} (1-t)^{k-1}\,dt\right)\nonumber\\
&&+(k-1)\frac{\partial}{\partial(k-1)}\left(\int_{0}^{1}t^{n-k} (1-t)^{k-1}\,dt\right)\Big]+(k-1)\log(n)\nonumber\\
&=& \log\left(\dfrac{1}{(k-1)!}\left(1\ldots \frac{n-k+1}{n}\right)\right)+(n-k)\frac{\partial}{\partial(n-k)}\log B(n,k)\nonumber\\
&&+(k-1)\left(\frac{\partial}{\partial(k-1)}\log B(n,k)+\log(n)\right)\nonumber
\end{eqnarray}
we know that $B(n,k)=Beta(n-k+1,k)=\frac{\Gamma(n-k+1)\Gamma(k)}{\Gamma(n+1)}$ then
\begin{eqnarray}
&=&\log\left(\dfrac{1}{(k-1)!}\Pi_{i=1}^{k}\left(1-\frac{i-1}{n}\right)\right)
+(n-k)\left(\frac{\partial \log(\Gamma(n-k+1))}{\partial(n-k)}-\frac{\partial \log(\Gamma(n+1))}{\partial(n-k)}\right)\nonumber\\
&&+(k-1)\left(\frac{\partial \log(\Gamma(k))}{\partial(k-1)}-\frac{\partial \log(\Gamma(n+1))}{\partial(k-1)}+\log(n)\right)\nonumber
\end{eqnarray}
from the fact that $\Gamma(\alpha+\beta)=\Pi_{i=1}^{\beta-1}(\alpha+i)\Gamma(\alpha+1)$ after some calculation we have
\begin{eqnarray}
&=&\log\left(\frac{1}{\Gamma(k)}\Pi_{i=1}^{k}\left(1-\frac{i-1}{n}\right)\right)-\sum_{i=1}^{k}\frac{n-k}{n-k+i}-(k-1)\left(\sum_{i=1}^{n-k+1}\frac{1}{k-1+i}-\log(n)\right)\nonumber\\
&=&\log\left(\frac{1}{\Gamma(k)}\Pi_{i=1}^{k}\left(1-\frac{i-1}{n}\right)\right)-\sum_{i=1}^{k}\frac{n-k}{n-k+i}-(k-1)\left(\sum_{j=1}^{n}\frac{1}{j}-\sum_{j=1}^{k-1}\frac{1}{j}-\log(n)\right).\nonumber
\end{eqnarray}
Taking limit with respect to $n$, we have
\begin{eqnarray}
    \lim_{n\to\infty}I_1(n)=-\log \Gamma(k)-k-(k-1)\left(\gamma-\sum_{j=1}^{k-1}\frac{1}{j}\right),\text{ for } k\geq 2\nonumber
\end{eqnarray}
where $\lim_{n\to\infty}\left(\sum_{j=1}^{n}\frac{1}{j}-\log(n)\right)=\gamma.$
\end{proof}

\begin{proof}[\textbf{Proof of Lemma \ref{Lemma.5}}]
From (\ref{Introduction_e2}) the pdf $g_n^{(k)}$ is
\begin{eqnarray}
g_n^{(k)}(x)&=&\dfrac{n!a_nf(a_nx+b_n)(n\overline{F}(a_nx+b_n))^{k-1}F^{n-1}(a_nx+b_n)}{(k-1)!n^{k-1}(n-k)!F^{k-1}(a_nx+b_n)}\nonumber\\
&=&\dfrac{n!g_n(x)(n\overline{F}(a_nx+b_n))^{k-1}}{(k-1)!n^{k}(n-k)!F^{k-1}(a_nx+b_n)}\nonumber
\end{eqnarray}
where $g_n(x)=na_nf(a_nx+b_n)F^{n-1}(a_nx+b_n)$ is the pdf of $F^n(a_nx+b_n).$ Under conditions Theorem \ref{gn_conv} and (\ref{Introduction_e1.2}) then
\begin{eqnarray}\label{order_density}
\lim_{n\to\infty}g_n^{(k)}(x)=
g(x)\dfrac{(-\log G(x))^{k-1}}{(k-1)!}
\end{eqnarray}
locally uniformly on $x\in \{y:\;G(y)>0\}.$
\end{proof}

\begin{proof}[\textbf{Proof of Lemma \ref{TPhi_orderEnt}}] 
\textbf{Case (i)} The entropy of $\phi_\alpha^{(k)}(x)$ is
\begin{eqnarray}
h(\phi_\alpha^{(k)})&=&-\int_{0}^{\infty}\dfrac{\alpha}{(k-1)!} x^{-\alpha k-1}e^{-x^{-\alpha}}\log\left(\dfrac{\alpha}{(k-1)!} x^{-\alpha k-1}e^{-x^{-\alpha}}\right)dx.\nonumber
\end{eqnarray}
Putting $x^{-\alpha}=u,$ we have $-\alpha x^{-\alpha-1}dx=du,$ and 
\begin{eqnarray}
h(\phi_\alpha^{(k)}) &=&-\int_{0}^{\infty}\dfrac{1}{(k-1)!} u^{k-1}e^{-u}\log\left(\dfrac{\alpha}{(k-1)!} u^{\frac{\alpha k+1}{\alpha}}e^{-u}\right)du\nonumber\\
&=&-\int_{0}^{\infty}\dfrac{1}{(k-1)!} u^{k-1}e^{-u}\log\dfrac{\alpha}{(k-1)!}du \nonumber\\
 &&-\dfrac{\alpha k+1}{\alpha(k-1)!}\int_{0}^{\infty}u^{k-1}e^{-u}\log u du +\int_{0}^{\infty}\dfrac{u^{k}e^{-u}}{(k-1)!}du \nonumber\\
&=&-\log\dfrac{\alpha}{(k-1)!}+I_A+\dfrac{1}{(k-1)!}\Gamma(k+1),\label{ent_g1}
\end{eqnarray}
where 
\begin{eqnarray}
I_A&=&-\dfrac{\alpha k+1}{\alpha(k-1)!}\int_{0}^{\infty}u^{k-1}e^{-u}\log u du =-\dfrac{\alpha k+1}{\alpha(k-1)!}A(k) \nonumber \\ 
&=&-\dfrac{\alpha k+1}{\alpha}\left(-\gamma+\sum_{i=1}^{k-1}\dfrac{1}{i}\right), \label{order_frechet_I1}
\end{eqnarray}
using Lemma \ref{Lemma3}. From (\ref{ent_g1}) and (\ref{order_frechet_I1}), 
\begin{eqnarray}
h(\phi_\alpha^{(k)})=-\log\dfrac{\alpha}{(k-1)!}-\dfrac{\alpha k+1}{\alpha}\left(-\gamma+\sum_{i=1}^{k-1}\dfrac{1}{i}\right)+\dfrac{\Gamma(k+1)}{(k-1)!}.\nonumber
\end{eqnarray}
\textbf{Case (ii)} 
The entropy of $\psi_\alpha^{(k)}(x)$ is
\begin{eqnarray}
h(\psi_\alpha^{(k)})&=&-\int_{-\infty}^{0}\dfrac{\alpha}{(k-1)!} (-x)^{\alpha k-1}e^{-(-x)^{\alpha}}\log\left(\dfrac{\alpha}{(k-1)!} (-x)^{\alpha k-1}e^{-(-x)^{\alpha}}\right)dx.\nonumber\\
&=&-\int_{0}^{\infty}\dfrac{1}{(k-1)!} u^{k-1}e^{-u}\log\left(\dfrac{\alpha}{(k-1)!} u^{\frac{\alpha k-1}{\alpha}}e^{-u}\right)du\nonumber
\end{eqnarray}
where $(-x)^{\alpha}=u,$ we have $-\alpha (-x)^{\alpha-1}dx=du$ 
\begin{eqnarray}
&=&-\int_{0}^{\infty}\dfrac{1}{(k-1)!} u^{k-1}e^{-u}\log\dfrac{\alpha}{(k-1)!}du\nonumber\\
 &&-\dfrac{\alpha k-1}{\alpha(k-1)!}\int_{0}^{\infty}u^{k-1}e^{-u}\log u du +\int_{0}^{\infty}\dfrac{u^{k}e^{-u}}{(k-1)!}du\nonumber\\
&=&-\log\dfrac{\alpha}{(k-1)!}+I_A+\dfrac{1}{(k-1)!}\Gamma(k+1),\nonumber
\end{eqnarray}
using Lemma \ref{Lemma3} we get 
\begin{eqnarray}
h(\psi_\alpha^{(k)})=-\log\dfrac{\alpha}{(k-1)!}-\dfrac{\alpha k-1}{\alpha}\left(-\gamma+\sum_{i=1}^{k-1}\dfrac{1}{i}\right)+\dfrac{\Gamma(k+1)}{(k-1)!}.\nonumber
\end{eqnarray}
\textbf{Case (iii)} 
The entropy of $\lambda^{(k)}(x)$ is
\begin{eqnarray}
h(\lambda^{(k)})&=&-\int_{-\infty}^{\infty}\dfrac{1}{(k-1)!} e^{-kx}e^{-e^{-x}}\log\left(\dfrac{1}{(k-1)!} e^{-kx}e^{-e^{-x}}\right)dx. \nonumber\\
 &=&\int_{0}^{\infty}\dfrac{1}{(k-1)!} u^{k-1}e^{-u}\log(k-1)!\;du \;(\text{where, } e^{-x}=u \text{ and } -e^{-x}dx=du)\nonumber\\
&&-\dfrac{k}{(k-1)!}\;\int_{0}^{\infty}u^{k-1}e^{-u}\log u\; du+\int_{0}^{\infty}\dfrac{u^{k}e^{-u}}{(k-1)!}du\nonumber\\
&=&\log(k-1)!+I_A+\dfrac{1}{(k-1)!}\Gamma(k+1),\nonumber
\end{eqnarray}
using Lemma \ref{Lemma3}, we get 
\begin{eqnarray}
h(\lambda^{(k)})=\log(k-1)!-k\left(-\gamma+\sum_{i=1}^{k-1}\dfrac{1}{i}\right)+\dfrac{\Gamma(k+1)}{(k-1)!}.\nonumber
\end{eqnarray}
\end{proof}

The method of proof the main theorem is similar the proof of Theorem 2.1, Saeb (2023b), by using the Helly-Bray theorem. If (\ref{Introduction_e1.3}) holds and $M_n^{(k)}=X_{n-k+1:n}$ for all bounded and continuous function $\mu$ on $[-v,v],$ for $v>0,$ then 
\begin{equation}
    \lim_{n\to\infty} E(\mu((M_n^{(k)}-b_n)a_n^{-1}))I_{\{|(M_n^{(k)}-b_n)a_n^{-1}|\leq v\}}=\int_{-v}^{v}\mu(x)dG^{(k)}(x).\label{HB}
\end{equation}
We need to concentrate on showing
$$\lim_{v\to\infty}\lim_{n\to\infty}E(\mu((M_n^{(k)}-b_n)a_n^{-1}))I_{\{|(M_n^{(k)}-b_n)a_n^{-1}|> v\}}=0$$
 for our problem. Now we state the proof of main theorem.


 \textbf{Proof of Theorem 2.1: Cases (a and c). }
 From (\ref{1.5})  we write,
	\begin{eqnarray}\label{F_Hgn}
	\lim_{n\to\infty}h(g_n^{(k)})=\log \Gamma(k)+k+(k-1)\left(\gamma-\sum_{i=1}^{k-1}\frac{1}{i}\right)+\lim_{n\to\infty}(I_A(n,v)+I_B(n,v)+I_C(n,v))\nonumber
	\end{eqnarray}
	where, $I_A(n,v)=-\int_{-\infty}^{\eta(v)}g_n^{(k)}(x)\log(a_n f(a_nx+b_n))dx,$
 $I_B(n,v)=-\int_{\eta(v)}^{v}g_n^{(k)}(x)$ \linebreak $\log(a_n f(a_nx+b_n))dx$
 and $I_C(n,v)=-\int_{v}^{\infty}g_n^{(k)}(x)\log(a_n f(a_nx+b_n))dx$
and $\eta(v)$ is any function which goes to $l(G)$ as $v\to\infty.$
	
We need to show that
\[\lim_{v\to\infty}\lim_{n\to\infty}(I_A(n,v)+I_C(n,v))=0.\]

First we set 
\[I_A(n,v)=-\int_{-\infty}^{\eta(v)}g_n^{(k)}(x)\log(a_n f(a_nx+b_n))dx.\]
Define the critical point $t_n=\frac{\xi_n-b_n}{a_n},$ which $\xi_n$ is given by $-\log F(\xi_n)\simeq n^{-1/2}.$  If $\frac{\xi_n-b_n}{a_n}\to c>0$ then $n^{1/2}\simeq n\log F(t_n\,a_n+b_n)\to \log(G(c))$ and this is contradict the fact that $n^{1/2}\to\infty.$ Therefore, $t_n\to l(G)$ as $\xi_n\to r(F)$ for large $n.$ We write,
	\begin{eqnarray}
	I_A(n,v)&=&-\int_{-\infty}^{t_n}g_n^{(k)}(x)\log(a_n f(a_nx+b_n))dx\nonumber\\
 &&-\int_{t_n}^{\eta(v)}g_n^{(k)}(x)\log(a_n f(a_nx+b_n))dx,\nonumber\\
	&=&I_{A_1}(n)+I_{A_2}(n,v).\label{F_IA}
	\end{eqnarray}
\textbf{Part I.} We write,
	\begin{eqnarray}
	I_{A_1}(n)&=&-\int_{-\infty}^{\xi_n}\log(a_n f(s))\frac{f(s)\,F^{n-k}(s)(n\bar{F}(s))^{k-1}}{n^{k-1}B(n,k)}ds\;\;(\text{where, } a_nx+b_n=s)\nonumber
	\end{eqnarray}
	Since, $0\leq F^{n-k}(s)(n\bar{F}(s))^{k-1}\leq n^{k-1}F^{n-k}(\xi_n)$ for all $s\in[-\infty,\xi_n]$  and $f(s)> 0$ for all $s,$ and use the mean value theorem, we get 
$$-\frac{F^{n-k}(\xi_n)}{B(n,k)}\left(\log(a_n)\int_{-\infty}^{\xi_n}f(s)ds+\int_{-\infty}^{\xi_n}f(s)\log(f(s))ds\right)<I_{A_1}(n)<0.$$
We assume that $h(f)<\infty,$ then,
	\begin{eqnarray}
	-\frac{F^{n-k}(\xi_n)}{B(n,k)}(\log(a_n)-h(f))<I_{A_1}(n)<0.\nonumber
	\end{eqnarray}
where $h(f)$ is entropy of $f.$	Since, $F(\xi_n)=\exp\{-n^{-1/2}\},$ from left hand side, we get  
	\begin{eqnarray}
	-\frac{F^{n-k}(\xi_n)}{B(n,k)}(\log(a_n)-h(f))&=&-\frac{\exp\{kn^{-1/2}\}}{B(n,k)\exp\{n^{1/2}\}}(\log(a_n)-h(f))\nonumber\\
	&<& I_{A_1}(n)<0.\label{FI_A1}
	\end{eqnarray}
		Without loss of generally, $a(n)\simeq a_n$ is a function of $n.$ If $G=\Phi_\alpha$ then $a(n)\in RV_{\frac{1}{\alpha}}$ (see, Resnick (1987), page 68). By Karamata representation in (\ref{kara}) $a(n)=c(n)\exp\Big\{\int_{1}^{n}\rho(t)
	t^{-1}dt\Big\}$ given $\epsilon>0$ there exists $N$ by $\lim_{n\to\infty}\rho(n)=\frac{1}{\alpha}$ such that $\frac{1-\epsilon}{\alpha}<\rho(n)<\frac{1+\epsilon}{\alpha},$ $n>N,$ so that 
	\begin{eqnarray}
	\left(\frac{1-\epsilon}{\alpha}\right)\log(n/N)<\int_N^n \frac{\rho(t)}{t}dt<\left(\frac{1+\epsilon}{\alpha}\right)\log(n/N).\nonumber
	\end{eqnarray}
	Therefore
	\begin{eqnarray}
	(n/N)^{\frac{1-\epsilon}{\alpha}}c<a(n)<(n/N)^{\frac{1+\epsilon}{\alpha}}c.\label{Fa_n}
	\end{eqnarray}
	
	Similarly, when $G=\Lambda$ we have $a(n)\in RV_0.$ Using Lemma \ref{slow}, for $n>N$ given $\epsilon>0$ such that $\abs{\rho(n)}<\epsilon$ then 
	\begin{eqnarray}
	(n/N)^{-\epsilon}c<a(n)< (n/N)^{\epsilon}c.\label{Fa_n2}
	\end{eqnarray}
Since $F\in\D(G),$ and from (\ref{FI_A1}) we get 
		\begin{eqnarray}
		\lim_{n\to\infty}I_{A_1}(n)=0.\label{F_IA1}
		\end{eqnarray}

\textbf{Part II.} Next, $I_{A_2}(n,v)=\int_{t_n}^{\eta(v)}\log(a_nf(a_nx+b_n))g_n^{(k)}(x)\,dx.$ 
In the case $G=\Phi_\alpha$ we have $\eta(v)=v^{-1}.$ Here we need to find out the dominating function for $g_n*{(k)}(x).$ 
\begin{eqnarray}
g_n^{(k)}(x)&=&\frac{a_nf(a_nx)}{n^{k-1}B(n,k)}\,F^{n-k}(a_nx)(n\bar{F}(a_nx))^{k-1},\nonumber\\
&=&\frac{a_nf(a_nx)}{n^{k-1}B(n,k)}\exp\Big\{\left(\frac{n-k}{n}\right)n\log F(a_nx)\Big\}(n\bar{F}(a_nx))^{k-1},\nonumber
\end{eqnarray}
from (\ref{von_F}) we note that, for given $\epsilon_1>0$ there exists $n_0$ such that $\frac{a_nxf(a_nx)}{-\log (F(a_nx))}\leq (\alpha+\epsilon_1),$ ultimately, therefore $a_nf(a_nx)\leq -(\alpha+\epsilon_1)\frac{\log (F(a_nx))}{x},$ for sufficient large $n.$ We have
\begin{eqnarray}
g_n^{(k)}(x)&<&(\alpha+\epsilon_1)\frac{(-n\log F(a_nx))^{k}}{x\,n^{k}B(n,k)}\exp\Big\{\frac{n-k}{n}n\log F(a_nx)\Big\},\nonumber
\end{eqnarray}
Theorem \ref{rv} show that, for sufficiently large $n$ given $\epsilon_2>0$ and  $n\bar{F}(a_nx)\simeq -n\log(F(a_nx))<\frac{1}{1-\epsilon_2}x^{-(\alpha-\epsilon_2)}.$ Finally, it is well known that, for all $n>n_0$ given $\epsilon_3>0,$ then $-\frac{n-k}{n}<-(1-\epsilon_3).$ Hence, for sufficiently large $n$ such that $n>n_0$ we write
\begin{eqnarray}
&<&\frac{(\alpha+\epsilon_1)}{(1-\epsilon_2)^k}\frac{x^{-1-k(\alpha-\epsilon_2)}}{(k-1)!}\exp\Big\{-\frac{1-\epsilon_3}{1-\epsilon_2}x^{-(\alpha-\epsilon_2)}\Big\},\nonumber\\
&<&\frac{c\alpha' x^{-k\alpha'-1}}{(k-1)!}\exp\Big\{-cx^{-\alpha'}\Big\}=\kappa(x).\label{F_h}
\end{eqnarray}
where, $\alpha'=\alpha-\epsilon_2,$ and $c$ is any positive constant. $\kappa$ is $k^{th}$ largest order statistics of Fr\'echet density function with scale and shape parameters.	For large $n,$ $g_n^{(k)}(x)< \kappa(x)$ and $$\int_{0}^{\infty}\log(\kappa(x))\kappa(x)dx<\infty.$$ We note that $\log(a_nf(a_nx))<\log(\kappa(x))-1.$ Using Lemma \ref{Lemma.5} $\lim_{n\to\infty}g_n^{(k)}(x)=\phi^{(k)}_\alpha(x),$ locally uniformly convergence in $x\in [0,\infty]$ by using dominated convergence theorem (DCT),
	\begin{eqnarray}
	\lim_{v\to\infty}\lim_{n\to\infty} I_{A_2}(n,v)=0.\label{F_IA2}
	\end{eqnarray}
	
	With similar argument, if $G=\Lambda$ then $\eta(v)=-v.$ The dominant function of $g_n(x)$ is given below. From Theorem \ref{lem3_appendix}, for any $\epsilon_1>0$ we have $f(a_nx+b_n)\leq -(1+\epsilon_1)\frac{\log F(a_nx+b_n)}{u(a_nx+b_n)}$ ultimately. We write,
		\begin{eqnarray}
		g_n^{(k)}(x)&=&\frac{a_nf(a_nx+b_n)}{n^{k-1}B(n,k)}F^{n-k}(a_nx+b_n)(n\bar{F}(a_nx+b_n))^{k-1},\nonumber\\
		&<&\frac{(1+\epsilon_1)}{n^{k}B(n,k)}\frac{u(b_n)}{u(a_nx+b_n)}(-n\log F(a_nx+b_n))^{k}\exp\Big\{(n-k)\log F(a_nx+b_n)\Big\}\label{g_n}
		\end{eqnarray}
	 Theorem \ref{lem4_appendix} show that, if $F\in\D(\Lambda)$ with auxiliary function $u$ and for any $\epsilon_2>0,$ there exists $n_0$ such that for all $x<0$ and $a_nx+b_n\geq n_0,$ \begin{eqnarray}
	 \frac{u(b_n)}{u(a_nx+b_n)}&<&\frac{1}{1-\epsilon_2}\left[\frac{-\log F(a_nx+b_n)}{-\log F(b_n)}\right]^{\epsilon_2},\nonumber
	 \end{eqnarray}
	 Since large $n$ and $\epsilon_2>0$ such that $(1-\epsilon_2)<-n\log(F(b_n))<(1+\epsilon_2)$ we write
	 \begin{eqnarray}
	 &<&\frac{\left(-n\log F(a_nx+b_n)\right)^{\epsilon_2}}{(1-\epsilon_2)^{\epsilon_2+1}},\nonumber
	 \end{eqnarray}
	 from (\ref{g_n}) we apply,	
		\begin{eqnarray}
			g_n{(k)}(x)&<&\frac{(1+\epsilon_1)(-n\log F(a_nx+b_n))^{\epsilon_2+k}}{n^{k}B(n,k)(1-\epsilon_2)^{(\epsilon_2+1)}}\exp\Big\{-(n-k)\log F(a_nx+b_n)\Big\},\nonumber
		\end{eqnarray}
		
	 From Lemma \ref{G_Lemma6} for any $\epsilon_3>0$ and $x<0$ we have $n\bar{F}(a_nx+b_n))<(1-\epsilon_3)^2(1+\epsilon_3\abs{x})^{\epsilon_3^{-1}},$
\begin{eqnarray}
&<&\frac{(1+\epsilon_1)(1-\epsilon_3)^{2(\epsilon_2+k)}}{(1-\epsilon_2)^{\epsilon_2+1}}(1+\epsilon_3\abs{x})^{\frac{k+\epsilon_2}{\epsilon_3}}\exp\Big\{-(1-\epsilon_3)^2(1+\epsilon_3\abs{x})^{\epsilon_3^{-1}}\left(\frac{n-k}{n}\right)\Big\},\nonumber\\
&=&c_1(1+\epsilon_3\abs{x})^{c_2}\exp\Big\{-c_3 (1+\epsilon_3\abs{x})^{\epsilon_3^{-1}}\Big\}=\kappa_1(x),\;\; \mbox{say.}\nonumber\\
\label{G_h}
\end{eqnarray}
where, $c_1=\frac{(1+\epsilon_1)(1-\epsilon_3)^{2(\epsilon_2+k)}}{(1-\epsilon_2)^{1+\epsilon_2}}$ and $c_2=\frac{k+\epsilon_2}{\epsilon_3},$ and $c_3=(1-\epsilon_3)^2.$\\
	Similarly, from (\ref{g_n}) and Theorem \ref{lem4_appendix} given $\epsilon_2>0$ and for $x>0$ such that $a_nx+b_n\geq n_0,$ $$\frac{u(b_n)}{u(a_nx+b_n)}<\frac{1}{1-\epsilon_2}\left[\frac{-\log F(b_n)}{-\log F(a_nx+b_n)}\right]^{\epsilon_2}<\frac{(1+\epsilon_2)^{\epsilon_2}}{1-\epsilon_2}\left(-n\log F(a_nx+b_n)\right)^{-\epsilon_2},$$
	from (\ref{g_n}) we write 
	\begin{eqnarray}
	g_n^{(k)}(x)&<&\frac{(1+\epsilon_1)}{n^{k}B(n,k)}\frac{u(b_n)}{u(a_nx+b_n)}(-n\log F(a_nx+b_n))^k\exp\Big\{-n\log F(a_nx+b_n)\Big\},\nonumber\\
	&<&\frac{(1+\epsilon_1)(1+\epsilon_2)^{\epsilon_2}}{(1-\epsilon_2)}(-n\log F(a_nx+b_n))^{k-\epsilon_2}\exp\Big\{-n\log F(a_nx+b_n)\Big\}.\nonumber
	\end{eqnarray}
	From Lemma \ref{G_Lemma6} for $\epsilon_3>0$ and $x>0$ we have $n\bar{F}(a_nx+b_n)<(1+\epsilon_3)^2(1+\epsilon_3x)^{\epsilon_3^{-1}},$ we have
	\begin{eqnarray}
	&<&\frac{(1+\epsilon_1)(1+\epsilon_3)^{2(k-\epsilon_2)}(1+\epsilon_2)^{\epsilon_2}}{(1-\epsilon_2)}(1+\epsilon_3 x)^{\frac{k-\epsilon_2}{\epsilon
			_3}}\exp\Big\{-(1+\epsilon_3)^2(1+\epsilon_3 x)^{\epsilon_3^{-1}}\Big\},\nonumber\\
	&=&c_1(1+\epsilon_3x)^{c_2}
	\exp\Big\{-c_3 (1+\epsilon_3 x)^{\epsilon_3^{-1}}\Big\}=\kappa_2(x),\;\; \mbox{say.}\nonumber\\
 \label{G_h3}
	\end{eqnarray}
	where, $c_1=\frac{(1+\epsilon_1)(1+\epsilon_3)^{2(k-\epsilon_2)}(1+\epsilon_2)^{\epsilon_2}}{(1-\epsilon_2)}$ and $c_2=\frac{k-\epsilon_2}{\epsilon
		_3},$ and $c_3=(1+\epsilon_3)^2.$
	From (\ref{G_h}) and (\ref{G_h3}) we get the general form of dominating function with
	\begin{eqnarray}\label{G_h2}
	\kappa(x)=c(1+\xi\abs{x})^{\frac{1}{\xi}-1}
	\exp\Big\{-c (1+\xi \abs{x})^{\xi^{-1}}\Big\}.
	\end{eqnarray}
	where, $c,$ and $\xi$ are any positive constants.
	
	 From (\ref{G_h2}) for large $n,$ we have $g_n^{(k)}(x)<\kappa(x),$ and $\int_{x<0}\kappa(x)\log(\kappa(x))dx<\infty$ (where, $\kappa$ is the form of Weibull density function for $x<0$). Using DCT
	\begin{eqnarray}
	\lim_{v\to\infty}\lim_{n\to\infty} I_{A_2}(n,v)=0.\label{G_IF2}
	\end{eqnarray}

	 From (\ref{F_IA}), if $F\in\D(G)$ then
	\begin{eqnarray}
	\lim_{v\to\infty}\lim_{n\to\infty}I_{A}(n,v)=0.\label{F_IA3}
	\end{eqnarray}

\textbf{Part III.} Next consider $G=\Phi_\alpha$ then $I_B(n,v)=-\int_{v^{-1}}^{v}\log (a_nf(a_n x))g_n(x)\,dx.$ From arguments in (\ref{F_h}) and Theorem \ref{TPhi_order}-i $\lim_{n\to\infty}g_n^{(k)}(x)=\phi^{(k)}_\alpha(x),$ locally uniformly convergence in $x\in [v^{-1},v]$ by using DCT and definition $g_n^{(k)}$ in (\ref{Introduction_e2}) we have
\begin{eqnarray}\label{F_IB}
\lim_{v\to\infty}\lim_{n\to\infty}I_B(n,v)&=&-\lim_{v\to\infty}\int_{-v}^{v}\lim_{n\to\infty}\log(g_n^{(k)}(x))g_n^{(k)}(x)dx-\lim_{n\to\infty}\log(n^{k-1}B(n,k))\nonumber\\
	&&+\lim_{v\to\infty}\int_{-v}^{v}\lim_{n\to\infty}\log(F^{n-k}(a_nx)(n\bar{F}(a_nx))^{k-1})g_n^{(k)}(x)dx\nonumber\\
	&=&h(\phi^{(k)}_\alpha)-\log\Gamma(k)+\int_{\Real}\log(\Phi_\alpha(x)x^{-\alpha(k-1)})\phi^{(k)}(x)dx\nonumber
 \end{eqnarray}
where $\lim_{n\to\infty}\log(n^{k-1}B(n,k))=\log\Gamma(k);$ Taking $u=x^{-\alpha}$ and after some calculation we get
 \begin{eqnarray}
 &=&h(\phi^{(k)}_\alpha)-\log\Gamma(k)-\frac{1}{(k-1)!}\int_{\Real}u^k e^{-u}du+\frac{(k-1)}{(k-1)!}\int_{\Real}u^{k-1} e^{-u}\;\log u\;du\nonumber\\
	&=&h(\phi^{(k)}_\alpha)-\log\Gamma(k)-k-(k-1)\left(\gamma-\sum_{i=1}^{k-1}\frac{1}{i}\right).
	\end{eqnarray}

In the case $G=\Lambda$ giving the dominated function in (\ref{G_h2}) and Theorem \ref{TPhi_order} for $x\in [-v,v]$ by using DCT, (with similar calculation in eq. \ref{F_IB})
	\begin{eqnarray}\label{G_IG}
\lim_{v\to\infty}\lim_{n\to\infty}I_B(n,v)=h(\lambda^{(k)})-\log\Gamma(k)-k-(k-1)\left(\gamma-\sum_{i=1}^{k-1}\frac{1}{i}\right).
	\end{eqnarray}

\textbf{Part IV.} Finally, $I_{C}(n,v)=-\int_{v}^{\infty}\log\left(\frac{ a_nf(a_n)}{\bar{F}(a_n)}\frac{f(a_nx+b_n)}{f(a_n)}\bar{F}(a_n)\right)\,g_n^{(k)}(x)dx$ by dividing and multiplying $f(a_n)$ and $\bar{F}(a_n).$ Since $f$ is positive nonincreasing function $\epsilon<\frac{f(a_nx+b_n)}{f(a_n)}\leq 1$ for all $x\geq 1$ and $0<\epsilon<1.$ It is trivially to show
	\begin{eqnarray}
	-\log (L(n))(1-F^n(a_nv+b_n))<I_{C}(n,v)<	-\log (L(n)\epsilon)(1-F^n(a_nv+b_n)),\nonumber
	\end{eqnarray}
	where, $L(n)=\dfrac{a_nf(a_n)}{\bar{F}(a_n)}n\bar{F}(a_n).$
From Theorem \ref{thm_von}, if $f$ is nonincreasing function and $F\in\D(\Phi_\alpha)$  then $L(n)\to  \alpha,$ as $n\to\infty.$ Also, if $F\in\D(\Lambda)$ then $\lim_{n\to\infty} L(n)=1.$ Since $\lim_{n\to\infty}F^n(a_nv+b_n)=G(v),$  we get
	\begin{eqnarray}
	\lim_{v\to \infty}\lim_{n\to\infty}I_{C}(n,v)=0.\label{F_IC}
	\end{eqnarray}

Finally, under assumptions (\ref{F_Hgn}), (\ref{F_IA3}) and (\ref{F_IC}) hold the proofs.

\textbf{Case (b).}
	Suppose $F \in \mathcal{D}(\Psi_\alpha),$ from Proposition 1.13 in Resnick (1987), $\bar{F}(r(F)-x^{-1})\in RV_{-\alpha}$ is regularly varying with index $-\alpha$. We set
	$a_n=1/(r(F)-\delta_n)$ and $\delta_n=F^{\leftarrow}(1 - \frac{1}{n}),$
	and then
	\[\lim_{n\to\infty}F^n(r(F)+a_n^{-1}x)=\Psi_\alpha(x),\;\;\text{ for }x<0.\]

	From Lemma \ref{rel_1}, $Z_i=1/(r(F)-X_i)$ is a iid rv with df $F_Z\in \D(\Phi_\alpha)$ and Lemma \ref{Lemma.5}, show that $$\lim_{n\to\infty}\tilde{g}_n^{(k)}(y)=\phi_\alpha^{(k)}(y),$$
	 where, $\tilde{g}_n^{(k)}(y)=\frac{(B(n,k))^{-1}}{a_ny^2n^{k-1}}f(r(F)-1/(a_ny))F^{n-k}(r(F)-1/(a_ny))(n\bar{F}(r(F)-1/(a_ny)))^{k-1}.$ From the definition of entropy, we write,
	\begin{eqnarray}
	h(\tilde{g}_n^{(k)})&=&2\int_{\epsilon_n}^{\infty}\log y \,\tilde{g}_n^{(k)}(y)dy\nonumber\\
 &&-\int_{\epsilon_n}^{\infty}\log\left(\frac{(B(n,k))^{-1}}{a_n n^{k-1}}f(y_n)F^{n-k}(y_n)(n\bar{F}(y_n))^{k-1}\right)\tilde{g}_n^{(k)}(y)dy\nonumber
	\end{eqnarray}
	where, $y_n=(r(F)-1/(a_ny))$ and  $\epsilon_n=a_n/(r(F)-l(F));$ We consider $y=-1/x,$ $x_n=r(F)+x/a_n$ we write
	\begin{eqnarray}
	&=&2E\left(\log\left(\vee_{i=1}^{n-k+1}Z_i a_n\right)
	I_{\{\epsilon_n\leq\vee_{i=1}^{n-k+1}Z_i a_n\}}\right)-\int_{\beta_n}^{0}\log\left(\tilde{g}_n^{(k)}(x)\right)\tilde{g}_n^{(k)}(x)dx\nonumber
	\end{eqnarray}
	where, $\beta_n=(l(F)-r(F))/a_n.$ Hence,
	\[h(g_n^{(k)})=h(\tilde{g}_n^{(k)})-2E\left(\log\left(\vee_{i=1}^{n-k+1}Z_i a_n\right)
	I_{\{\epsilon_n\leq \vee_{i=1}^{n-k+1}Z_i a_n\}}\right).\]
	According to lemma \ref{Lemma3.1}, $Y_i=\alpha\log Z_i$ is iid rv with $F_Y\in\D(\Lambda)$ then
\begin{eqnarray}	\label{exp2}
\lim_{n\to\infty}h(g_n^{(k)})=\lim_{n\to\infty}h(\tilde{g}_n^{(k)})-\lim_{n\to\infty}\frac{2}{\alpha}E\left(\vee_{i=1}^{n-k+1}Y_i+\alpha\log a_n\right)I_{\{\alpha\log\epsilon_n\leq \vee_{i=1}^{n-k+1}Y_i+\alpha\log a_n\}}.\nonumber\\
 \end{eqnarray}
With similar argument from the moment convergence theorem (Theorem \ref{Moment}-(iii)), it trivial to show that for $r(F)<\infty$
	\begin{eqnarray}
	E(\abs{Y_1}I_{\{Y_1\leq 0\}})&=&E\left(\abs{\log(r(F)-X_1)}I_{\{1\leq(r(F)-X_1<r(F))\}}\right)\nonumber\\
	&\stackrel{<}{\text{(Jensen's inequality)}}&\left(\abs{\log E(r(F)-X_1)I_{\{0\leq X_1\leq r(F)\}}}\right)<\infty,\nonumber
	\end{eqnarray}
 then
 \begin{eqnarray}\label{Exp1}
     \lim_{n\to\infty}E\left(\vee_{i=1}^{n-k+1}Y_i+\alpha\log a_n\right)=\int_{\Real}y\lambda^{(k)}(y)dy=-\left(-\gamma+\sum_{i=1}^{k-1}\frac{1}{i}\right).
 \end{eqnarray}
	From (\ref{exp2}) and (\ref{Exp1}), under conditions Theorems 2.1 (a) and , we get
	\begin{eqnarray}	\lim_{n\to\infty}h(g_n^{(k)})&=&h(\phi^{(k)})+\frac{2}{\alpha}\left(-\gamma+\sum_{i=1}^{k-1}\frac{1}{i}\right)\nonumber\\
	&=&h(\psi_\alpha^{(k)}).\nonumber
	\end{eqnarray}

\hfill $\square$ 

\appendix
\section[Appendix A]{}\label{more results}

\begin{rem}\label{kara1}(Remark, Page 19, Resnick (1987))
	If $U\in RV_\rho$ then $U$ has representation
	\begin{eqnarray}
	U(x)=c(x)\exp\Big\{\int_{1}^{x}t^{-1} \rho(t)dt\Big\},\label{kara}
	\end{eqnarray}
	where, $\lim_{x\to\infty}c(x)=c$ and $\lim_{t\to\infty}\rho(t)=\rho.$
\end{rem}

\begin{lem}\label{slow}(Corollary, Page 17, Resnick (1987))
	$L$ is slowly varying iff $L$ can be represented as
	\begin{eqnarray}
	L(x)=c(x)\exp\Big\{\int_{1}^{x}t^{-1} \epsilon(t)dt\Big\},
	\end{eqnarray} for $x>0$ and $\lim_{x\to\infty}c(x)=c$ and $\lim_{t\to\infty}\epsilon(t)=0.$
\end{lem}


\begin{thm}\label{a_rv}(Proposition 0.12, Resnick (1987))
	If $V\in\Pi$ with auxiliary function $a(t)$ then $a(\cdot)\in RV_0.$
\end{thm}


\begin{thm}(Proposition 0.8, Resnick (1987))\label{rv}
	Suppose $U\in RV_{\rho},$ $\rho\in\Real.$ Take $\epsilon>0.$ Then there exists $t_0$ such that for $x\geq 1$  and $t\geq t_0$
	\[(1-\epsilon)x^{\rho-\epsilon}<\frac{U(tx)}{U(t)}<(1+\epsilon)x^{\rho+\epsilon}.\]
\end{thm}

\begin{thm}(Proposition 2.1, Resnick (1987))\label{Moment}
	Let $F\in\mathcal{D}(G).$
	\begin{enumerate}
		\item[(i)] If $G=\Phi_\alpha,$ then $a_n=\left(1/(1-F)\right)^{\leftarrow}(n),$ $b_n=0,$ and if for some integer $0<k<\alpha,$
		\[\int_{-\infty}^{0}\mid{x}\mid^kF(dx)<\infty,\]
		then $\lim_{n\to\infty}E\left(\dfrac{M_n}{a_n}\right)^k=\int_{0}^{\infty}x^k\Phi_\alpha(dx)=\Gamma\left(1-\frac{k}{\alpha}\right).$
		\item[(ii)] If $G=\Psi_\alpha,$ then $a_n=r(F)-\left(1/(1-F)\right)^{\leftarrow}(n),$ $b_n=r(F),$ and if for some integer $k>0,$
		\[\int_{-\infty}^{r(F)}\mid{x}\mid^kF(dx)<\infty,\]
		then $\lim_{n\to\infty}E\left(\dfrac{M_n-r(F)}{a_n}\right)^k=\int_{-\infty}^{0}x^k\Psi_\alpha(dx)=(-1)^k\Gamma\left(1+\frac{k}{\alpha}\right).$
		\item[(iii)] If $G=\Lambda,$ then $b_n=\left(1/(1-F)\right)^{\leftarrow}(n),$ $a_n=u(b_n),$ and if for some integer $k>0,$
		\[\int_{-\infty}^{0}\mid{x}\mid^kF(dx)<\infty,\]
		then $\lim_{n\to\infty}E\left(\dfrac{M_n-b_n}{a_n}\right)^k=\int_{-\infty}^{\infty}x^k\Lambda(dx)=(-1)^k\Gamma^{(k)}(1),$
		where $\Gamma^{(k)}(1)$ is the k-th derivative of the Gamma function at $x=1.$
	\end{enumerate}
\end{thm}

\begin{thm}(Proposition 1.15 and 1.16, Resnick (1987)) \label{thm_von}
	\item[(i)] Suppose that df $F$ is absolutely continuous with positive density $f$ in some neighborhood of $\infty.$ \\
	\item[(a)] If for some $\alpha>0$
	\begin{eqnarray}\label{von_F}
	\lim_{x\to\infty}\dfrac{xf(x)}{\overline{F}(x)}=\alpha
	\end{eqnarray}
	then $F\in\D(\Phi_\alpha).$\\
	\item[(b)] If $f$ is nonincreasing and $F\in\D(\Phi_\alpha)$ then (\ref{von_F}) holds.
	\item[(ii)] Suppose $F$ has finite right endpoint $r(F)$ and is absolutely continuous in a left neighborhood of $r(F)$ with positive density $f.$  \\
	\item[(a)] If for some $\alpha>0$
	\begin{eqnarray}\label{von_W}
	\lim_{x\to r(F)}\dfrac{(r(F)-x)f(x)}{\overline{F}(x)}=\alpha
	\end{eqnarray}
	then $F\in\D(\Psi_\alpha).$\\
	\item[(b)] If $f$ is nonincreasing and $F\in\D(\Psi_\alpha)$ then (\ref{von_W}) holds.
\end{thm}




\begin{thm} (Proposition 1.17, Resnick (1987)) Let  $F$ be absolutely continuous in a left neighborhood of $r(F)$ with density $f.$ If 
	\begin{eqnarray} \label{lem3_appendix} \lim_{x\uparrow r(F)}f(x)\int_{x}^{r(F)}\overline{F}(t)dt/\overline{F}(x)^2=1,\end{eqnarray} 
	then $F\in\D(\Lambda).$ In this case we may take,
	\[u(t)=\int_{x}^{r(F)}\overline{F}(s)ds/\overline{F}(t), \;\;b_n=F^{\leftarrow}(1-1/n),\;\;a_n=u(b_n).\]
\end{thm}
\begin{thm}\label{lem4_appendix}(Lemma 2, De Haan and Resnick (1982))
	Suppose $F\in \D(\Lambda)$ with auxiliary function $u$ and $\epsilon>0.$ There exists a $t_0$ such that for $x\geq 0,$ $t\geq t_0$
	\begin{eqnarray}
	(1-\epsilon)\left[\frac{-\log F(t)}{-\log F(t+xu(t))}\right]^{-\epsilon}\leq \frac{u(t+xu(t))}{u(t)}\leq (1+\epsilon)\left[\frac{-\log F(t)}{-\log F(t+xu(t))}\right]^{\epsilon},\nonumber
	\end{eqnarray}
	and for $x<0,$ $t+xu(t)\geq t_0$
	\begin{eqnarray}
	(1-\epsilon)\left[\frac{-\log F(t+xu(t))}{-\log F(t)}\right]^{-\epsilon}\leq \frac{u(t+xu(t))}{u(t)}\leq (1+\epsilon)\left[\frac{-\log F(t+xu(t))}{-\log F(t)}\right]^{\epsilon}.\nonumber
	\end{eqnarray}
\end{thm}
\begin{thm}(Theorem 2.5, Resnick (1987))\label{gn_conv}
	Suppose that $F$ is absolutely continuous with pdf $f.$ If $F\in\D(G)$ and
	\begin{itemize}
		\item[(i)] $G=\Phi_\alpha,$ then $g_n(x)\to \phi_\alpha(x)$ locally uniformly on $(0,\infty)$ iff (\ref{von_F}) holds;
		\item[(ii)] $G=\Psi_\alpha,$ then $g_n(x)\to \psi_\alpha(x)$ locally uniformly on $(-\infty,0)$ iff (\ref{von_W}) holds;
		\item[(iii)] $G=\Lambda,$ then $g_n(x)\to \lambda(x)$ locally uniformly on $\Real$ iff (\ref{lem3_appendix}) holds.
	\end{itemize}
\end{thm}

\begin{thm}\label{Lem5_appendix}(Corollary, Balkema and De Haan (1972))
	A distribution function $F\in\D(\Lambda)$ if and only if there exist a positive function $c$ satisfying $\lim_{x\to r(F)}c(x)=1$ and a positive differentiable function $u(t)$ satisfying $\lim_{x\to r(F)}u'(x)=0$ such that
	\[\bar{F}(x)=c(x)\exp\Big\{-\int_{-\infty}^{x}\frac{dt}{u(t)}\Big\}\;\;\text{for }x<r(F).\]
\end{thm}

\begin{lem}(Lemma A.2, Saeb (2023) )\label{G_Lemma6}
	Suppose $F\in\D(\Lambda)$ with auxiliary function $u$ and $\epsilon>0.$ There exists a large $N$ such that for $x\geq 0,$ and $n>N$
	\begin{eqnarray}
	n(1-F(a_nx+b_n))\leq(1+\epsilon)^2(1+\epsilon x)^{\epsilon^{-1}},\nonumber
	\end{eqnarray}
	and for $x<0,$
	\begin{eqnarray}
	n(1-F(a_nx+b_n))\leq(1-\epsilon)^2(1+\epsilon \abs{x})^{\epsilon^{-1}}.\nonumber
	\end{eqnarray}
\end{lem}

\begin{lem} (Lemma A. 3, Saeb (2023)) Let $X_1, X_2,\ldots$ are iid rvs with common df $F.$ Suppose $Z_i=1/(r(F)-X_i)$ with common df $F_Z$ and $Y_i=\alpha\log X_i$ with common df $F_Y.$
\begin{enumerate}
\item[(i)] \label{rel_1}  If $F\in\D(\Psi_\alpha)$ with norming constant $a_n>0$ and $b_n=r(F)$ then $F_Z\in\D(\Phi_\alpha)$ with $\tilde{a}_n=1/a_n$ and $\tilde{b}_n=0.$

\item[(ii)] \label{Lemma3.1}
	 If $F\in\D(\Phi_\alpha)$ with $a_n=F^{\leftarrow}\left(1-\frac{1}{n}\right)$ and $b_n=0$ then $F_Y\in \D(\Lambda)$ with $a_n^*=1$ and $b_n^*=\alpha\log a_n.$
\end{enumerate}
\end{lem}


\begin{lem}\label{Lemma3} The value of the integral
\begin{eqnarray}
A(k)=\int_{0}^{\infty} u^{k-1}e^{-u}\log u du=(k-1)!\left(-\gamma+\sum_{i=1}^{k-1}\dfrac{1}{i}\right), \; k\geq 2,
\end{eqnarray} with $\;A(1) = - \gamma.\;$
\end{lem}
\begin{proof}
For $k=1,$ we have 
\begin{eqnarray}
A(1)&=&\int_{0}^{\infty} e^{-u}\;\log u \;du=-\gamma.
\end{eqnarray}
For $k=2,$ we have 
\begin{eqnarray}
A(2) &=& \int_{0}^{\infty} u\;e^{-u}\log u\; du \nonumber \\ 
&=&\int_{0}^{\infty}\log u\;e^{-u}\;du+\int_{0}^{\infty}e^{-u}du, \text{using integration by parts,} \nonumber\\
&=&A(1)+\Gamma(1)=1-\gamma.\nonumber
\end{eqnarray}
Assuming the result for arbitrary $k-1,$ we have 
\begin{eqnarray}\label{lemma3_A}
A(k-1)&=& (k-2)!\left(-\gamma+\sum_{i=1}^{k-2}\dfrac{1}{i}\right),\nonumber\\
&=&-\gamma(k-2)!+\dfrac{(k-2)!}{k-2}+(k-2)\dfrac{(k-3)!}{k-3}+[(k-2)(k-3)]\dfrac{(k-4)!}{k-4}\nonumber\\
&&+\ldots+[(k-2)(k-3)(k-4)\ldots 5]\dfrac{4!}{4}+[(k-2)(k-3)\ldots 4]\dfrac{3!}{3}+(k-2)!,\nonumber\\
&=&-\gamma(k-2)!+\Gamma(k-2)+(k-2)\Gamma(k-3)+[(k-2)(k-3)]\Gamma(k-4),\nonumber\\
&&+\ldots+[(k-2)(k-3)(k-4)\ldots 5]\Gamma(4)+[(k-2)(k-3)\ldots 4]\Gamma(3)\nonumber\\
&&+[(k-2)(k-3)\ldots 3]\Gamma(2),\nonumber\\
&=&(k-2)A(k-2)+\Gamma(k-2).
\end{eqnarray}
We then have, 
\begin{eqnarray}
A(k)&=&\int_{0}^{\infty}u^{k-1}\log u e^{-u}\;du, \nonumber \\ 
&=&(k-1)\int_{0}^{\infty}u^{k-2}\log u\;e^{-u}du+\int_{0}^{\infty}u^{k-2}e^{-u}\;du,\nonumber\\
&=&(k-1)A(k-1)+\Gamma(k-1), \nonumber \\
&=&(k-1)!\left(-\gamma+\sum_{i=1}^{k-2}\dfrac{1}{i}\right)+\Gamma(k-1) \;\text{using}\; (\ref{lemma3_A}),\nonumber\\
&=&-\gamma(k-1)!+\dfrac{(k-1)!}{k-1}+(k-1)\dfrac{(k-2)!}{k-2}+[(k-1)(k-2)]\dfrac{(k-3)!}{k-3}\nonumber\\
&&+\ldots+[(k-1)(k-2)(k-3)\ldots 4]\dfrac{3!}{3}+[(k-1)(k-2)\ldots 3]\dfrac{2!}{2}+(k-1)!,\nonumber\\
&=&(k-1)!(-\gamma+\dfrac{1}{k-1}+\dfrac{1}{k-2}+\dfrac{1}{k-3}+\ldots+\dfrac{1}{3}+\dfrac{1}{2}+1),\nonumber\\
&=&(k-1)!\left(-\gamma+\sum_{i=1}^{k-1}\dfrac{1}{i}\right),\;k\geq 2.\nonumber
\end{eqnarray} Hence, by induction, the proof is complete. 
\end{proof}


\end{document}